\documentclass[12pt]{article}
\usepackage{amsmath,amssymb,amsthm, amsfonts}
\usepackage{etaremune, enumerate, float}
\usepackage{hyperref}
\usepackage{graphicx}
\usepackage{url}
\usepackage[mathlines]{lineno}
\usepackage{dsfont} 
\usepackage{tikz, graphicx, subcaption, caption}
\usepackage[algo2e,ruled,vlined]{algorithm2e}
\usepackage{color}
\definecolor{red}{rgb}{1,0,0}

\definecolor{blue}{rgb}{0,0,1}

\definecolor{green}{rgb}{0,.6,0}

\numberwithin{figure}{section}   

\newtheorem{thm}{Theorem}[section]
\newtheorem{cor}[thm]{Corollary}
\newtheorem{lem}[thm]{Lemma}
\newtheorem{prop}[thm]{Proposition}

\newtheorem{obs}[thm]{Observation}

\newtheorem{quest}[thm]{Question}

\theoremstyle{definition}
\newtheorem{rem}[thm]{Remark}

\theoremstyle{definition}

\theoremstyle{definition}
\newtheorem{ex}[thm]{Example}

\def\incpp#1{G(\mathcal{P}_{#1})}

\newcommand{\rad}{\operatorname{rad}} 

\newcommand{\Z}{\operatorname{Z}}
\newcommand{\Zp}{\operatorname{Z}_+}

\newcommand{\thr}{\operatorname{th}}
\newcommand{\thp}{\operatorname{th}_+}
\newcommand{\thc}{\operatorname{th}_c}

\newcommand{\pt}{\operatorname{pt}}
\newcommand{\ptp}{\operatorname{pt}_+}

\newcommand{\capt}{\operatorname{capt}} 

\newcommand{\dist}{\operatorname{d}}  
\newcommand{\len}{\operatorname{len}}  

\newcommand{\bit}{\begin{itemize}}
\newcommand{\eit}{\end{itemize}}
\newcommand{\ben}{\begin{enumerate}}
\newcommand{\een}{\end{enumerate}}
\newcommand{\beq}{\begin{equation}}
\newcommand{\eeq}{\end{equation}}
\newcommand{\bea}{\begin{eqnarray*}} 
\newcommand{\eea}{\end{eqnarray*}}
\newcommand{\bpf}{\begin{proof}}
\newcommand{\epf}{\end{proof}\ms}
\newcommand{\bmt}{\begin{bmatrix}}
\newcommand{\emt}{\end{bmatrix}}
\newcommand{\ms}{\medskip}

\newcommand{\lc}{\left\lceil}
\newcommand{\rc}{\right\rceil}
\newcommand{\lf}{\left\lfloor}
\newcommand{\rf}{\right\rfloor}
\newcommand{\du}{\,\dot{\cup}\,}
\newcommand{\noi}{\noindent}

\title{Throttling for the game of Cops and Robbers  on graphs}
\author{Jane Breen\thanks{Department of Mathematics, University of Manitoba, Winnipeg, MB R3T 2N2, Canada, breenj3@myumanitoba.ca.}\and Boris Brimkov\thanks{Department of Computational and Applied Mathematics, Rice University, Houston, TX 77005, USA,  boris.brimkov@rice.edu.}\and Joshua Carlson\thanks{Department of Mathematics, Iowa State University, Ames, IA 50011, USA, (jmsdg7, hogben, reinh196)@iastate.edu.}\and Leslie Hogben\footnotemark[3]\ \thanks{American Institute of Mathematics, 600 E. Brokaw Road, San Jose, CA 95112, USA, hogben@aimath.org} 
\and K.E. Perry\thanks{Department of Mathematics, University of Denver, Denver, CO 80208, USA,  Katherine.E.Perry@du.edu.}\and Carolyn Reinhart\footnotemark[3]}

\begin{document}
\maketitle

\vspace{-10pt}

\begin{abstract} 
We consider the cop-throttling number of a graph $G$ for the game of Cops and Robbers, which is defined to be the minimum of $(k + \capt_k(G))$, where $k$ is the number of cops and $\capt_k(G)$ is the minimum number of rounds needed for $k$ cops to capture the robber on $G$ over all possible games. We provide  some 
tools for bounding the cop-throttling number, including showing that the positive semidefinite (PSD) throttling number, a variant of zero forcing throttling, is an upper bound for the cop-throttling number. We also characterize graphs having low cop-throttling number and investigate how large the cop-throttling number can be for a given graph. We consider trees, unicyclic graphs, incidence graphs of finite projective planes (a Meyniel extremal family of graphs), a family of cop-win graphs with maximum capture time, grids, and hypercubes. All the upper bounds on the cop-throttling number we obtain for families of graphs are  $ O(\sqrt n)$.  \end{abstract}

\noi {\bf Keywords} Cops and Robbers, propagation time, throttling, zero forcing, positive semidefinite

\noi{\bf AMS subject classification} 05C57, 05C15, 05C50


\section{Introduction}\label{sintro} 

Cops and Robbers, first introduced in \cite{AF84, NW83, Q78}, is a two-player pursuit-evasion game played on a simple\footnote{In the literature, a reflexive graph, i.e., a graph in which each vertex has a loop, is often used but this is equivalent to using a simple graph in which staying in place is considered a move.} graph.  The game has two players: one who places and moves a collection of cops and the other who places and moves a single robber. The goal for the cops is to catch the robber by occupying the same vertex the robber occupies; this is called a {\em capture}.  The goal of the robber is to avoid capture.  After an initial placement of the cops on a multiset of vertices (meaning more than one cop can occupy a single vertex), followed by the placement of the robber, the game is played over a sequence of time-steps or rounds during which the players take turns, both playing in a single round, starting with the cops, as follows: The team of cops takes a turn by allowing each cop one chance to move to an adjacent vertex or decide to stay in place. Similarly, the robber takes a turn by moving to an adjacent vertex or deciding to stay in place. The cops win the game if after some finite number of rounds, a cop captures the robber.  If the robber has a strategy to evade the cops indefinitely, the robber wins.

Much of the study of the game of Cops and Robbers has focused on the {\em cop number} $c(G)$ of a graph $G$, first introduced in \cite{AF84}, which is the minimum number of cops required to capture the robber, regardless of the robber's strategy. A graph is called {\em cop-win} if $c(G) = 1$. Establishing an upper bound for this parameter is the focus of {\em Meyniel's Conjecture} \cite{frankl1987cops}, which states that if $G$ is a connected graph of order $n$, then $c(G) = O(\sqrt{n})$.

The {\em length} of a game is the number of rounds needed for the cops to catch the robber and  a game is played {\em optimally} if its length is the minimum over all possible strategies for the cops. Note that length does not include the initial placement of the cops and robber and assumes that the robber is playing to avoid capture for as long as possible. Given a set of $k$ cops, define the {\em k-capture time}, $\capt_k(G)$, to be the smallest number of rounds that $k$ cops need to win on $G$ over all possible games. We say that $\capt_k(G)$ equals infinity if the robber can evade capture with $k$ cops. This invariant was introduced in \cite{BGHK09} with $k = c(G)$, (in this case we simply write $\capt(G)$), and most recently studied in \cite{BPPR17} with $c(G) \leq k \leq \gamma(G)$, where $\gamma(G)$ denotes the domination number of $G$.

In this paper, we consider the situation in which interest is placed on optimizing the sum of the resources used to accomplish a task and the time needed to complete that task. To that end, we define the {\em cop-throttling number}, $\thc(G)$, of a graph $G$ to be 
\[ \thc(G) =\min_{k \in [n]} \left(k+\capt_k(G)\right)\vspace{-4pt} \]
where $n$ is the order of $G$ and $[n]=\{1,\dots,n\}$. Observe that if $k \geq n$,  $\capt_k(G) = 0$, and if $\gamma(G) \leq k < n$, $\capt_k(G) = 1$. Further, if $k < c(G)$, $\capt_k(G)$ is infinity. Thus, it makes sense to restrict our consideration of this invariant to the range $c(G) \leq k \leq \gamma(G)$. 

We remark here that the idea of throttling first appeared and was studied by Butler and Young in \cite{BY13throttling} for zero forcing {on} a graph. Zero forcing is a coloring game played on a graph in which vertices are initially colored blue or white, white vertices can be colored blue using a color change rule, and the goal is to color all of the vertices blue (see Section \ref{sPSDZF} for precise definitions concerning zero forcing). The {\em standard zero forcing  number}, $\Z(G)$,  is the minimum number of initially blue vertices needed to color the entire graph blue using the standard color change rule. The standard propagation time of $G$ is the minimum number of time-steps needed to color  all the vertices of $G$ blue starting with  $\Z(G)$ blue vertices, when all independent color changes are performed simultaneously at each time-step.  The definitions of standard zero forcing number and standard propagation time are analogous to cop number and capture time, respectively. Butler and Young define the throttling number as the  minimum of the sum of $|S|$ and the propagation time of $S$ over sets $S$ of vertices. It was this study, and a subsequent study of throttling for positive semidefinite zero forcing \cite{PSDthrottle}, that led to our interest in throttling for the game of Cops and Robbers.

In the present work, we investigate throttling for the game of Cops and Robbers on graphs. We begin in Section \ref{sbasic} with some basic observations and examples, as well as a  useful clique sum result. In Section \ref{sPSDZF} we show that the cop-throttling number of a tree (respectively, unicyclic graph) of order $n$ is at most $2\sqrt n$ (respectively, $\sqrt 6 \sqrt n$). These bounds are established by showing that the positive semidefinite (PSD) throttling number  is an upper bound for the cop-throttling number of a graph; in fact, we show that the two are equal for trees.  We also discuss connections between cop-throttling and other graph parameters in this section. 
In Section \ref{sextremeL} we characterize graphs having low cop-throttling number and in Section \ref{high}, we investigate how large $\thc(G)$ can be for a given graph $G$. We consider incidence graphs of finite projective planes (a Meyniel extremal family of graphs), a family of cop-win graphs with maximum capture time, grids, and hypercubes. Interestingly, all the upper bounds on the cop-throttling number we obtain for families of connected graphs are  $ O(\sqrt n)$, where $n$ is the order of $G$.  If the result $\thc(G)= O(\sqrt n)$ were established in general, it would imply Meyniel's Conjecture, since $c(G)\le \thc(G)$.

Lastly, we remark here that cop-win graphs, which arise naturally throughout the study of Cops and Robbers, have been characterized \cite{NW83, Q78}:  A {\em corner} in a graph $G$ is a vertex $u$ such that there is another vertex $v$ with $N[u]\subseteq N[v]$. Every cop-win graph has a corner, and deletion of corners is used to reduce a cop-win graph to a single vertex: A graph is {\em dismantlable} if there is an ordering of the vertices $v_1,\dots,v_{n}$ such that $v_k$ is a corner in $G-\{v_1,\dots,v_{k-1}\}$ for $k=1,\dots,n-1$.
\begin{thm}\label{thm:copwin} {\rm \cite{NW83, Q78}, \cite[Theorem 2.3]{CRbook}} A graph $G$ is cop-win if and only if $G$ is dismantlable.
\end{thm} 
For more information on the game of Cops and Robbers, see \cite{CRbook}; we follow most of the notation in that book. For additional graph theoretic terminology and notation, see \cite{GTbook}.


\section{Tools for bounding cop-throttling number}\label{sbasic}

This section contains some elementary observations about 
cop-throttling, as well as a result on clique sums of graphs. The observations and clique sum result are used as tools in obtaining bounds for the cop-throttling number of graphs in later sections.

As mentioned in the introduction, when investigating cop-throttling, we often restrict the number of cops by assuming $c(G) \leq k \leq \gamma(G)$ and notice that when $k = \gamma(G)$, $\capt_k(G) = 1$. The following observation is a direct consequence.

\begin{obs}\label{thc-dom} For every graph $G$,  $\thc(G)\le \gamma(G)+1$.
\end{obs}


The bound in Observation \ref{thc-dom} is tight. For example, consider the complete graph $K_n$, which has $\thc(K_n)=2=\gamma(K_n)+1$.  However, the gap between $\thc(G)$ and $\gamma(G)+1$ can be arbitrarily large, as shown in the next example.

\begin{ex}\label{ex:SW} Let $m\ge 3$.  A wheel $W_{m+1}$ on $m+1$ vertices  is constructed by adding a dominating vertex to a cycle $C_m$; the vertex of degree $m$ is called the {\em center} and the other vertices are called {\em cycle vertices}.  Define the {\em stellated wheel} $SW_{2m+1}$ to be the graph constructed from  $W_{m+1}$ by adding $m$ new vertices with each new vertex adjacent to two consecutive cycle vertices, resulting in a graph of order $2m+1$.  The graph $SW_{21}$ 
is  shown in Figure \ref{fig:notdomex}.  We see that $\gamma(SW_{2m+1})=\lc\frac m 2\rc$ because choosing every other cycle vertex dominates the graph, and a dominating set must contain a neighbor of every degree two vertex.  We claim that $\thc(SW_{2m+1})=3$, showing that the cop-throttling number can be arbitrarily less than domination number.  By placing one cop on the center, the robber can be caught  in at most two rounds, so $\thc(SW_{2m+1})\le 3$.  That $\thc(SW_{2m+1})\ge 3$ is clear from the fact that  
 $\capt_k(SW_{2m+1}) \geq 2$ when $\lc \frac m 2 \rc>k $. (See also Proposition \ref{thc=3}, which characterizes graphs $G$ with $\thc(G)=3$.)\vspace{-4pt}

\begin{figure}[h!] \begin{center}
\scalebox{.4}{\includegraphics{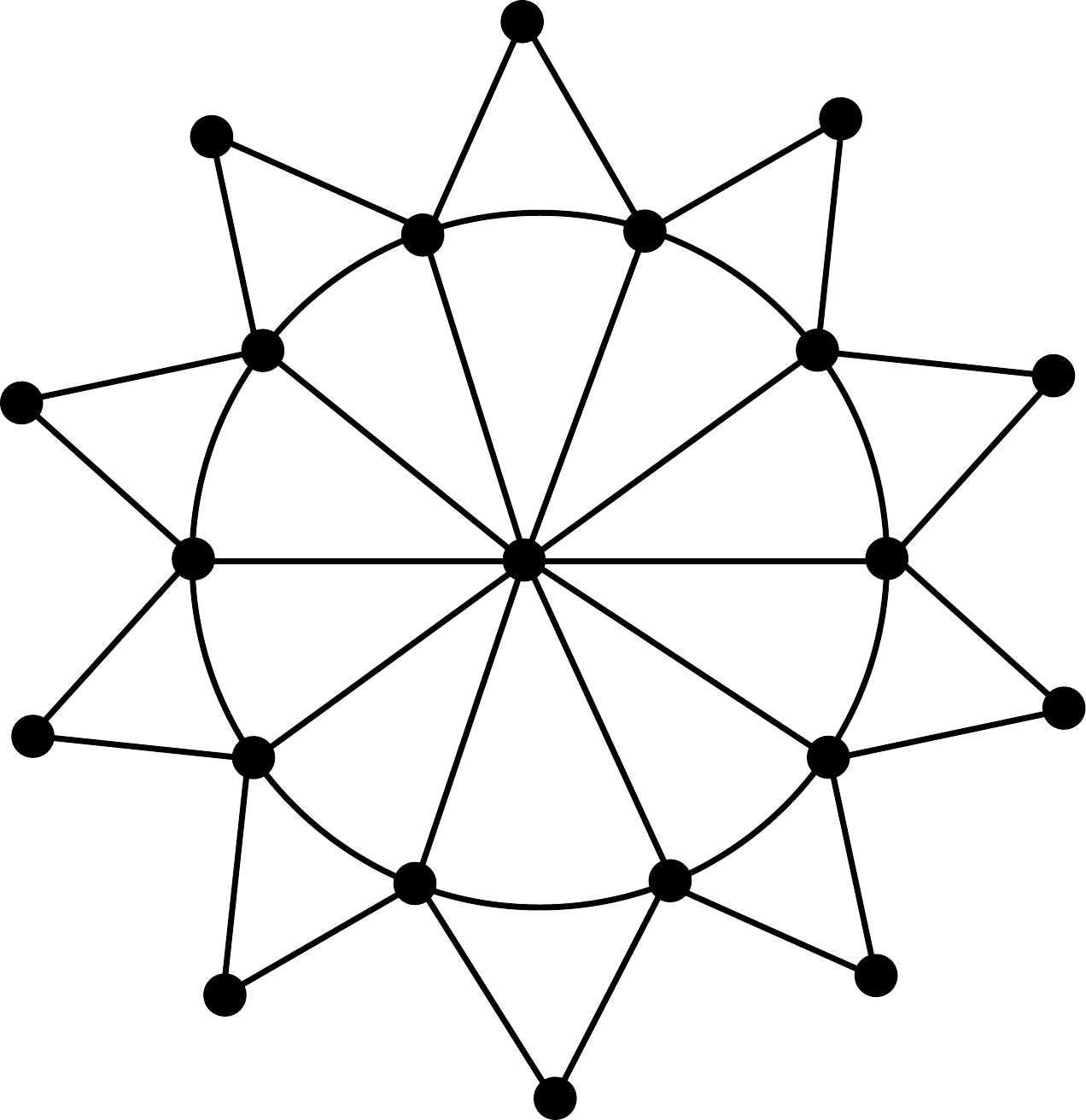}}
\caption{\label{fig:notdomex} The graph $SW_{21}$  \vspace{-15pt}}
\end{center}
\end{figure}
\end{ex}

Recall that an induced subgraph $H$ of $G$ is a \emph{retract} of a graph $G$ if there is a mapping $\varphi:G\to H$ whose restriction to $H$ is the identity and such that $uv\in E(G)$ implies $\varphi(u)\varphi(v)\in E(H)$ or $\varphi(u)=\varphi(v)$. 
Retracts are an important tool in the game of Cops and Robbers. The following theorem was proved in \cite{BPPR17} for the capture time with $k$ cops on a graph. 

\begin{thm} {\rm \cite{BPPR17}} \label{capt_retract}
Suppose that $V(G) = V_1 \cup V_2 \cup \cdots \cup V_t$ where $G[V_i]$ is a retract of $G$ for each $i$, and $k = \sum_{i = 1}^t k_i$. Then\vspace{-4pt} \[\capt_k(G) \leq \max_{1 \leq i \leq t} \capt_{k_i}(G[V_i]).\vspace{-4pt}\]
\end{thm}

The proof of the above is achieved by specifying a strategy for the cops so that the robber is always captured within a certain number of rounds, producing an upper bound for the capture time. In particular, each of the $t$ teams of cops plays using the \emph{shadow strategy} on its retract: The $k_i$ cops in $G[V_i]$ chase the image of the robber under the retract without leaving that subgraph. After $\max_{i} \capt_{k_i}(G[V_i])$ rounds have passed, every team of cops has caught their image of the robber, with at least one team capturing the actual robber.

It is also possible to determine a lower bound for the capture time by specifying a strategy for the robber in which capture occurs within a certain number of rounds. We will use this observation and the above to establish bounds on the cop-throttling numbers of clique sums of graphs.

Suppose that $G$ and $H$ are graphs such that $V(G) \cap V(H) = K_m$, the complete graph on $m$ vertices (also known as a {\em clique}), and each of $G$ and $H$ has order greater than $m$. Then $G\cup H$ is called the {\em clique sum of $G$ and $H$}. We can consider clique sums as a method for constructing new graphs from old; if $G$ and $H$ both contain a copy of $K_m$ then we can identify those vertices to form a new graph. In the special case when $m = 1$, $G\cup H$ is called the {\em vertex sum of $G$ and $H$ at $v$}, where $V(G) \cap V(H) = \{v\}$.

\begin{thm} \label{cliquesum}
Let $G_1$ and $G_2$ be graphs with $G_1\cap G_2=K_m$, and let $G=G_1\cup G_2$. Let $k_1$, $k_2$, $p_1$, and $p_2$ be numbers such that for $i\in \{1,2\}$, $k_i$ cops can always catch a robber in $G_i$ in $p_i$ rounds, and $k_i+p_i=\thc(G_i)$. Then, \vspace{-4pt}
\[
\max\{k_1+p_1,k_2+p_2\}\leq \thc(G)\leq k_1+k_2+\max\{p_1,p_2\},\vspace{-4pt}
\]
and both bounds are tight.
\end{thm}

\begin{proof}  We begin by establishing the upper bound.  
To see that  both $G_1$ and $G_2$ are retracts of $G$, consider  a graph mapping $G_1 \cup G_2 \to G_1$ (respectively, $G_2$) that sends all vertices of $V(G_2)\setminus V(G_1)$ (respectively, $V(G_1)\setminus V(G_2)$) to a single vertex in the intersection. Then, by playing with $k_1$ cops in $G_1$ and $k_2$ cops in $G_2$,  the capture time with $k_1 + k_2$ cops in $G$ is bounded above by $\max\{p_1, p_2\}$ by Theorem \ref{capt_retract}. Hence, the upper bound holds.

For the lower bound, we specify a strategy for the robber in which the robber starts and remains inside $V(G_1)$ until caught (and that this is known to the cops). We claim that for this strategy of the robber, $k_1$ cops will be sufficient to catch the robber in $p_1$ rounds. Indeed, no advantage is offered to the cops by the existence of vertices outside $V(G_1)$, since a cop starting on a vertex in $V(G_2)\setminus V(G_1)$ receives no advantage compared to  starting in $K_m$. 
Since $k_1$ cops can catch the robber in $p_1$ rounds in $G[V(G_1)]$, and, by a similar argument, $k_2$ cops can catch the robber in $p_2$ rounds in $G[V(G_2)]$, it follows that $\max\{\thc(G_1),\thc(G_2)\}\leq \thc(G)$.

For tightness of the upper bound, let $G_1$ be the graph obtained by identifying a leaf $\ell$ of $P_4$ with a vertex of the complete graph $K_m$. Then, $\thc(G_1)=3$, since there is no dominating vertex in $G_1$, and since a cop placed at the neighbor of $\ell$ in $P_4$ can always catch the robber in two rounds. From the point of view of an upper bound, $(k,p)=(1,2)$ is better than $(k,p)=(2,1)$, which is the only other possibility to achieve $\thc(G_1)=3$.  
Let $G_2\simeq G_1$, and let $G$ be the graph obtained by identifying the copy of $K_m$ in $G_1$ to the copy of $K_m$ in $G_2$. Then, $k_1=k_2=1$ and $p_1=p_2=2$, and we claim that $\thc(G)=4=k_1+k_2+\max\{p_1,p_2\}$. Indeed, two cops placed at the two vertices which are at distance 1 from $K_m$ can catch the robber in two rounds; moreover, $G$ has no dominating set of size 2, nor is there a vertex of $G$ which is at distance at most 2 from all other vertices. Thus, $\thc(G)=4$. 

For tightness of the lower bound, let $G_1$ be the graph obtained by appending a leaf to a vertex $v$ of $K_m$; then, $\thc(G_1)=2$ since $v$ is a dominating vertex. Let $G_2\simeq G_1$, and let $G$ be the graph obtained by identifying the copy of $K_m$ in $G_1$ to the copy of $K_m$ in $G_2$ in such a way that the degree $m$ vertex in $G_1$ is identified with the degree $m$ vertex in $G_2$. Then, $\thc(G)=2=\max\{\thc(G_1),\thc(G_2)\}$, since the degree $m+1$ vertex of $G$ is a dominating vertex.
\end{proof}

The next corollary is immediate and will be used in Section \ref{high}.
\begin{cor} \label{cor:cliquesum}
Let $G_1$ and $G_2$ be graphs with $G_1\cap G_2=K_m$, and let $G=G_1\cup G_2$.  Then, 
$
 \thc(G)\leq \thc(G_1)+\thc(G_2).$
\end{cor}


\section{Connections with other graph parameters and bounds for trees and unicyclic graphs}\label{sPSDZF} 

This section discusses connections between Cops and Robbers  and  positive semidefinite (PSD) zero forcing, applies results on PSD throttling to determine the cop-throttling numbers of paths and cycles, and establishes $O(\sqrt n)$ upper bounds for the cop-throttling and PSD throttling numbers of trees and unicyclic graphs.  We also discuss connections of cop-throttling to girth and burning number.

  First, we describe standard zero forcing and its throttling number in more detail. 
 The {\em standard color change rule} \cite{AIM08} allows changing the color of a white vertex $w$ to blue  when $w$ is the only white neighbor of a blue  vertex $v$; in this case we say that $v$ {\em forces} $w$. 
A subset $S$ of vertices that when colored blue initially can eventually force all vertices of $G$ is called a {\em standard zero forcing set}. In contrast to Cops and Robbers, in zero forcing $S$ is a set rather than a multiset, since a multiset confers no advantage over a set.
Starting with the vertices in $S$ blue, the number of time-steps required for this process to color all vertices blue (performing all possible independent forces at each  step) is the {\em standard propagation time} of the set $S$, denoted by $\pt(G,S)$ \cite{proptime}.  Butler and Young \cite{BY13throttling}  define $\thr(G,S)=|S|+\pt(G,S)$ for  $S\subseteq V(G)$ and define the \textit{throttling number} of $G$ as $\thr(G)=\min \{ \thr(G,S)\!:\! S\text{ is a zero forcing set}\}$. 

PSD  throttling uses PSD zero forcing,  which was introduced in \cite{smallparam} to derive an upper bound on the maximum nullity of a positive semidefinite matrix whose pattern of nonzero off-diagonal entries is described by the graph. For  a set $S$  of blue vertices, let $W_1,...,W_k$ be the sets of white vertices corresponding to the $k$ connected components of $G-S$  (it is possible that $k=1$). 
The \textit{PSD color change rule} allows changing the color of $w_i\in W_i$ from white to blue  when $w_i$ is the only white neighbor of $v$ in $G[W_i \cup S]$; in this case we say that $v$ {\em forces} $w_i$.  
A subset $S$ of  vertices that when   colored blue initially can eventually force all vertices of $G$ under the PSD color change rule is called a {\em PSD zero forcing set}. 
The minimum cardinality of a PSD zero forcing set for $G$ is the {\em PSD zero forcing number} of $G$ and is denoted by $\Z_+(G)$.  
Starting with the vertices in $S$ blue, the number of time-steps required for this process to color all vertices blue (performing all possible independent forces at each  step) is the {\em PSD propagation time} of set $S$, denoted by $\ptp(G;S)$;  if $S$ is not a PSD zero forcing set, then $\ptp(G;S)$ is infinity.   The {\em PSD propagation time} of graph $G$ is $\ptp(G)=\min\{\ptp(G;S)\!:\! S\mbox{ is a minimum PSD zero forcing set of }G\}$ \cite{PSDpropTime}.  In \cite{PSDthrottle} Carlson et al.~define $\thp(G;S)=|S| + \ptp(G;S)$ and the {\em PSD throttling number} of a graph $G$   as  %
$\thp(G) = \min_{S\subseteq V(G)} \thp(G;S) $.  

Since we want to exhibit a parallel between cop-throttling and PSD throttling, we  define the {\em capture time for $S\subseteq V(G)$}, denoted by $\capt(G;S)$, to be the maximum   of the number of rounds needed to catch the robber (over all robber placements) when the cops are placed on the elements of the multiset $S$.  We say $S\subseteq V(G)$ is a {\em capture set} for  $G$ if placing the cops on the vertices in $S$ ensures that the robber can be caught regardless of the strategy used.  If $S$ is not a capture set for $G$,  then $\capt(G;S)$ is infinity.   Observe that the  $k$-capture time of $G$ is
$\capt_k(G)=\min_{|S|=k}  \capt(G;S)$.  
  For $S\subseteq V(G)$, we define  $\thc(G;S)=|S| + \capt(G;S)$ and note that 
$ \thc(G) = \min_{S\subseteq V(G)} \thc(G;S).$


\begin{thm}\label{ptc-le-ptp}  Let $S\subseteq V(G)$ be a PSD  zero forcing set.  Then, $S$ is a capture set, $\capt(G;S)\le \ptp(G;S)$,   $\thc(G;S)\le \thp(G;S)$,  
$c(G)\le \Zp(G)$, 
and $\thc(G)\le \thp(G)$.
\end{thm}
\bpf
We convert a PSD  forcing  process  for $S$ into a robber capture process (set of moves  for the cops) using the same number of time-steps as follows:  First place the cops on the vertices in   $S$.  The robber then chooses a location.  At each time-step, the cops move only into the component of $G-S$ containing the robber and ignore the other components; thus, the location of each cop is uniquely determined at every time-step.  We show  that at any stage of PSD  zero forcing, the robber can never get to any  vertex colored blue  (regardless of whether it contains a cop) without passing through a vertex occupied by a cop (a {\em cop vertex}) and therefore getting captured.  Thus, coloring the component containing the robber blue is sufficient to catch the robber.  

 Initially, there is a cop on each  blue vertex (the vertices of $S$), so clearly the robber cannot access a blue vertex without being captured.  Now, assume we have a set $B$ of blue vertices obtained by performing robber catching moves as described.
Suppose the robber is on a white vertex $w$ and can move onto a blue vertex $v$ that is not a cop vertex.  Then $w$ is adjacent to $v$,  and $v$ was formerly occupied by a cop who moved to another vertex $u$  in the same component containing $w$ by performing a PSD force.  This contradicts the PSD forcing rule because $v$ was adjacent to two white vertices ($u$ and $w$) in the same component when $v$ forced $u$. Thus, the robber can never get to any blue vertex without   getting caught, so the robber will indeed be caught through this process.  Therefore, $S$ is a capture set and $\capt(G;S)\le\ptp(G;S)$.   The remaining statements are immediate consequences of this.   \epf


 For a positive integer $k$ and a graph $G=(V,E)$,  the {\em $k$-center radius} is \vspace{-3pt}
\[\rad_k(G)=\min_{S\subseteq V, |S|=k}\max_{v\in V}\dist(v,S).\vspace{-3pt}\]
Note that $\rad_1(G)=\rad(G)$ (the usual radius of $G$), and any vertex $x$ such that\break $\max_{v\in V}\dist(v,\{x\})=\rad(G)$ is called a {\em center}.   It is shown in  \cite[Lemma 2]{BPPR17} that $\capt_k(G)\ge \rad_k(G)$ for any graph $G$.  This result and Theorem \ref{ptc-le-ptp} imply that 
$\ptp(G;S)\ge \rad_{|S|}(G)$ for any graph $G$  (this is is also clear from the definitions).  As noted in  \cite[Corollary 2]{BPPR17}, $\capt_k(T)= \rad_k(T)$ for any tree $T$.

\begin{rem}\label{kradth} For any graph $G$ and $1\le k \le \gamma(G)$, $\min_k(k+\rad_k(G))\le \thc(G)$.   For any tree $T$, 
$\thc(T)=\min_k(k+\rad_k(T))$.  Note that this leads to a polynomial time algorithm for the cop-throttling number of a tree, since the $k$-radius of a tree can be computed in polynomial time  \cite{MT}.
\end{rem}


\subsection{Throttling for Cops and Robbers on trees}

This section   applies results on PSD throttling to determine the cop-throttling numbers of paths  and to establish $2\sqrt n$ as an upper bound for the cop-throttling and PSD throttling numbers of trees.

\begin{thm}\label{tree-ptc-eq-ptp} Suppose $T$ is a tree.  Then for $S\subseteq V(T)$,  $\capt(T;S)=\ptp(T;S)$ and  $\thc(T;S)=\thp(T;S)$.  Furthermore, $\thc(T)=\thp(T)$.
\end{thm}

\bpf  
It is shown in the proof of  \cite[Lemma 2]{BPPR17} that 
$\capt(T;S)\ge \max_{v\in V(T)}\dist(v,S)$. Since every vertex of degree at least two in  a tree is a cut vertex, it follows  from the definitions of PSD zero forcing and PSD propagation time of a set that $\ptp(T;S)=\max_{v\in V(T)}\dist(v,S)$.  Thus, $\ptp(T;S)\le \capt(T;S)$.  In combination  with Theorem \ref{ptc-le-ptp}, this implies $\ptp(T;S)= \capt(T;S)$. 
The remaining statements now follow.
  \epf


Since for a tree $c(T)=1=\Zp(T)$, 
we see that   $\capt(T)=\ptp(T)=\rad(T)$.
  While a graph $G$ has $\Zp(G)=1$ if and only if $G$ is a tree, there are many additional graphs that have $c(G)=1$.    

The next result follows from Theorem \ref{tree-ptc-eq-ptp} and \cite[Theorem 3.2]{PSDthrottle}.
\begin{cor}  \label{thmThrP}
Let $n\geq 1$. Then \[\thc (P_n)= \lc \sqrt{2n}-\frac{1}{2} \rc.
\]
\end{cor}

We also obtain the following corollaries for trees from \cite[Section 4]{PSDthrottle}.

\begin{cor} \label{noLeaf}
Let $T$ be a tree on $n\geq 3$ vertices. Then there exists a set $S\subseteq V(T)$  that  contains no leaves  such that
$\thc(T)=\thc(T;S).  $
\end{cor}

\begin{cor}\label{subtree}
If $T$ is a tree with subtree $T'$, then
\[\thc(T')\le\thc(T).  \]
That is, the cop-throttling number is subtree monotone for trees.
\end{cor}

\begin{cor}\label{treeBound1}
If $T$ is a tree with radius $r$, then \[\lc{\sqrt{2(2r+1)}-\frac{1}{2}}\rc\leq\thc(T)\leq r +1.\] 
\end{cor}

A \emph{full binary tree of height $h$}, denoted by $T_B(h)$, is a tree with {\em root} vertex $z\in V(T)$ such that $\deg(z)=2$, $\dist(v,z) \leq h$ for all $ v\in V(T)$, and  $\displaystyle{ \deg(v)=\begin{cases} 3 &\dist(v, z) <h \\ 1 &\dist(v,z) =h. \end{cases}  }$ for all  $v\neq z$. 

\begin{cor}\label{binary}
For the full binary tree of height $h$, $\thc(T_B(h))=h+1$.
\end{cor}

Observe that if we denote the order of $T_B(h)$ by $n$, then $\thc(T_B(h))=h+1=\Theta(\log n)$.  It is shown in \cite{PSDthrottle} that  $\thp(G)=\Omega(\log n)$ for a graph $G$ of order $n$.  By contrast,  $\thc(P_n)=\Theta(\sqrt n)$, and 
we now show that  $\thc(T)=O(\sqrt n)$ for all  trees $T$ of order $n$.  This is established  by providing an algorithm to construct a set $S$ with $|S|\le \lf \sqrt n\rf$ and $\ptp(S)=\lf\sqrt n\rf$.\vspace{8pt}

\begin{algorithm2e}[H]{ 
\caption{For a tree $T$ of order $n$,  return PSD zero forcing set  $S$ with $|S|\le \lf \sqrt n\rf$ and $\ptp(T;S)\le \lf \sqrt n\rf$.\label{alg2sqrtnT}}
\vspace{1mm}
\textbf{given} $T=(V,E)$\;
 $n \gets |V|$\;
 $t\gets \lf \sqrt n\rf$\;
 $S \gets \emptyset$\;
 $T' \gets T$\;
 $c\gets$ a center vertex of $T'$\;
 $r\gets \rad(T')$\;
\While{$r> t$}{
 Find a vertex $v$ such that $\dist_{T'}(v,c)=r$\; 
 Find the vertex $u$ on the path from $v$ to $c$ such that $\dist_{T'}(v,u)=t$\; 
$S \gets S \cup \{u\}$\;
 $T' \gets$  the component of $T'-u$ that contains $c$\;
 $c\gets $  a center vertex of $T'$\;
 $r\gets \rad(T')$\;
}

$S \gets S \cup \{c\}$\;
\Return $S$.\vspace{1mm}

}\end{algorithm2e}

\begin{thm}\label{thcT2sqrtn} Algorithm $\ref{alg2sqrtnT}$ finds a set $S$ such that $|S|\le \lf\sqrt n\rf$, $\capt(T;S)=\ptp(T;S)\le \lf\sqrt n\rf$, and $\thc(T;S)= \thp(T;S)\le 2 \lf\sqrt n\rf$.  Thus, $\thc(T)=\thp(T)\le 2\lf\sqrt n\rf$.
\end{thm}

\bpf   We show that $|S|\le \lf\sqrt n\rf$ and $\ptp(T;S)\le \lf\sqrt n\rf$.  These results, together with Theorem \ref{tree-ptc-eq-ptp}, imply the remaining statements.  As in the algorithm, let $t=\lf\sqrt n\rf$.  

First  we show  that $|S| \le t$.   
Each iteration of the while-loop adds one vertex $u$ to $S$.  If at most $t-2$ iterations of the while-loop are executed, then $|S|\le t-2+1=t-1$.  So assume at least $t-1$ iterations of the while-loop are executed.  Each iteration of the while-loop  removes at least $t+1$ vertices from $T'$, so after $t-1$ iterations,  $|V(T')|\le n-(t-1)(t+1)=n-t^2+1$.  Since $t=\lf\sqrt n\rf$, $n< (t+1)^2=t^2+2t+1$.  So $|V(T')|\le 2t+1$, which implies $\rad(T')\le t$ and the 
while-loop 
terminates.  Thus, $|S|=t-1+1=t$.

To complete the proof,  we show  that $\ptp(T;S) \le t$.  
Consider a  vertex $u$  added to $S$  in the while-loop,  and let $w$ be a vertex 
of $T'$ 
outside the component of $T'-u$ containing $c$ (a center vertex of $T'$).  Then  $\dist_{T'}(u,w)\le t $, so $u$ can force $w$ in at most $t$ time-steps.  When  vertex $c$ is  added to $S$  in the last step before returning, $\dist_{T'}(u,x)\le t $ for every  $x\in V(T')$, so $c$ can force $w$ in at most $t$ time-steps.  
\epf

 We do not have examples showing the factor of 2 in the  bound in Theorem \ref{thcT2sqrtn} is  tight. 
However, by examining the path, we know that the scalar multiple of $\sqrt n$ in any upper  bound for $\thc(T)$  must be at least $\sqrt 2$  (cf. Corollary \ref{thmThrP}).     The next example shows that there exists a tree having a larger cop-throttling number  than the path with the same order.

\begin{ex}  Let $T$ be the tree shown in Figure \ref{fig:pathnotbiggesttree}. 
Observe that $\rad_1(T)=4, \rad_2(T)=3, \rad_3(T)=2$, and $\rad_4(T)=1$.  Since $\thc(T)=\min_k (k+\rad_k(T))$,  \[\thc(T)= 5>4=\thc(P_{10}).\]
\end{ex}
\begin{figure}[h!] \begin{center}
\scalebox{.4}{\includegraphics{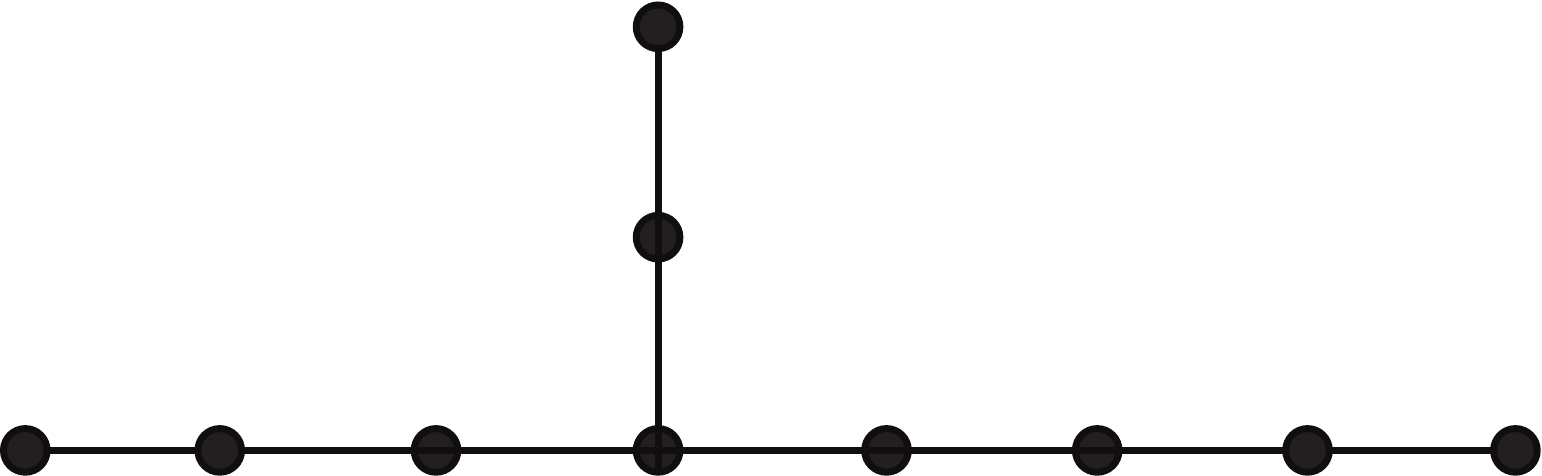}}
\caption{\label{fig:pathnotbiggesttree} A tree of order $10$ for which $\thc(T)=5>4=\thc(P_{10})$ \vspace{-10pt} }
\end{center}
\end{figure}

Other similar examples can be constructed, although we have not found a family of trees  whose cop-throttling numbers can be arbitrarily larger than the cop-throttling numbers of paths with the same order.


\subsection{Throttling for Cops and Robbers on unicyclic graphs}

This section   applies results on PSD throttling to determine the cop-throttling number of cycles and  to establish $\sqrt 6\sqrt n$ as an upper bound for the cop-throttling and PSD throttling numbers of unicyclic graphs.
First we consider cycles.  

\begin{rem}\label{rem:cycle} Note that $c(C_n)=2$ for $n\ge 4$, 
and any placement set $S$ of cops with $|S|\ge 2$ allows PSD forcing following the movement of cops as if the robber were in each segment.  Thus $\capt(C_n;S)=\ptp(C_n;S)$,   $\thc(C_n;S)=\thp(C_n;S)$, and $\thc (C_n)=\thp (C_n)$.  Furthermore, $\thc(C_n)=\min_k(k+\rad_k(C_n))$.  
\end{rem}
The next result now follows from \cite[Theorem 3.3]{PSDthrottle}.

\begin{thm}  \label{thmThrC}
Let $n\geq 4$. 
Then \[\thc (C_n)=\thp (C_n)=  \lc \sqrt{2n}-\frac{1}{2}   \rc.
\]
\end{thm}

The equality of cop-throttling and PSD-throttling does not hold for all unicyclic graphs, as the next example shows.

\begin{ex}  Let $G$ be the graph shown in Figure \ref{fig:unicyclicex}.  Since $G$ does not have a corner, 
we must use at least two cops (blue vertices) to catch the robber (force the graph).  
Since $\gamma(G)=4$, we cannot achieve four or less as the value of throttling number (for Cops and Robbers or PSD zero forcing) with propagation time equal to one.  These observations imply  $\thc(G)\ge 4$.

    
To complete the argument that $\thc(G)=4$, we show that $\capt(G;\{8,9\})=2$. Note  that every vertex is within distance 2 of a cop.  If the robber is on the cycle,  the cop on vertex 8 can immediately move in the direction of the robber since the cop on vertex 9 is guarding vertex 4.  Thus, $\capt(G;\{8,9\})=2$ and  $\thc(G)=4$.  

Next we show that $\thp(G)=5$.  Since $\ptp(G;\{8,9\})=3$, $\thp(G)\le 5$. Suppose $S$ is a PSD zero forcing set with $|S|=2$. By previous remarks, to show that $\thp(G)=5$, it suffices to show that $\ptp(G;S)\ge 3$.  In order for vertex 11 to be blue after time-step 2, necessarily one of the vertices  9, 10, 11 must be in $S$.  No forcing can take place on the cycle until there are at least two blue vertices on the cycle, which cannot happen until after time-step 1, after which it takes at least two more time-steps to force the cycle.  Thus, $\ptp(G;S)\ge 3$. 

\begin{figure}[h!] \begin{center}
\scalebox{.45}{\includegraphics{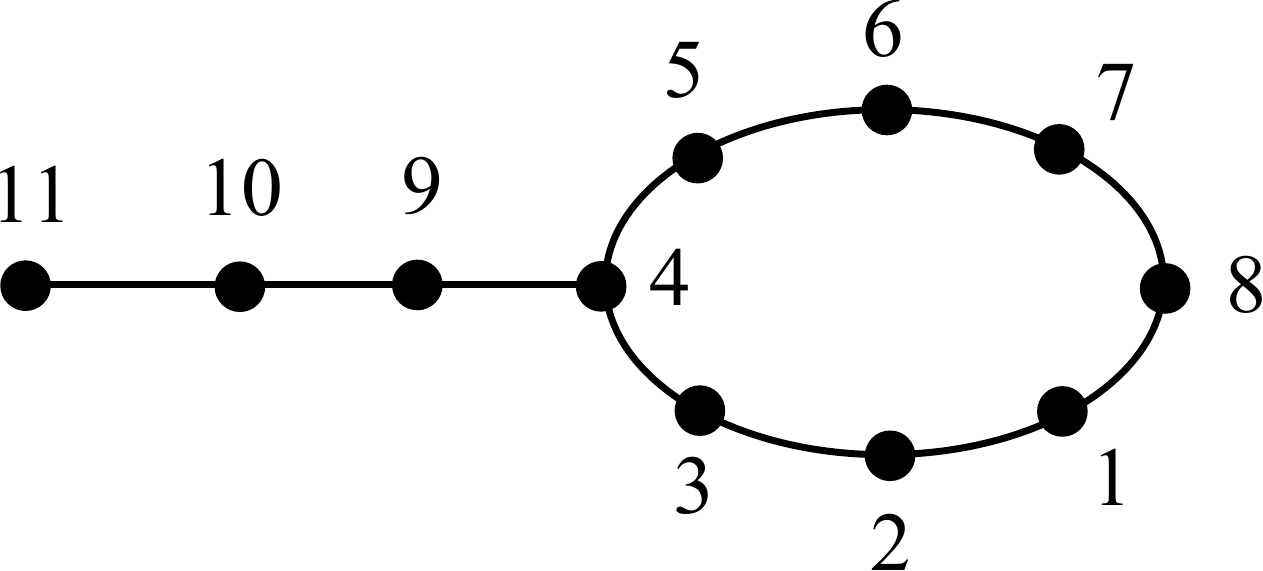}}
\caption{\label{fig:unicyclicex}A unicyclic graph for which $\thc(G)=4<5=\thp(G)$ \vspace{-10pt}}
\end{center}
\end{figure}
\end{ex}

We can use Theorem \ref{thcT2sqrtn} and Theorem \ref{thmThrC} to obtain an $O(\sqrt n)$ upper bound on  the PSD-throttling number of a unicyclic graph. This is established in Theorem \ref{thcU3sqrtn} below, after a modification of Algorithm \ref{alg2sqrtnT}.

\begin{rem}\label{thcT2sqrtn-mod} We can modify Algorithm \ref{alg2sqrtnT} slightly in the case that certain vertices are known to be blue (this situation arises in the case of a  forest  attached to  a cycle, where we use a strategy to first force the cycle).   Let $X\subseteq V(T)$ and assume that these vertices are blue before applying the algorithm. We then modify  Algorithm \ref{alg2sqrtnT} by measuring distance  not counting a vertex in $X$, and using $n'=n-|X|$.  This is illustrated in the next example.
 The reason we can ignore a vertex in $X$ in counting distances  is that when forcing is started in the tree, that vertex will start forcing immediately, so distances will be shorter. 
 Modified Algorithm \ref{alg2sqrtnT}  finds a set $S\supseteq X$ such that   $|S|\le |X|+\sqrt {n'}$ and $\ptp(T;S)\le\sqrt {n'}$.  We observe that adding fixed blue vertices for no reason is not desirable as it raises the number of blue vertices at least as much as it reduces propagation time, but it can improve propagation time slightly when the vertices are known to be blue anyway.  This situation in which it is known that  some vertices must be blue has been studied for both standard zero forcing and power domination in \cite{restrict}. \end{rem}

\begin{ex}
Consider $P_{9}$ with vertices $1,\dots,9$ in path order.  Algorithm \ref{alg2sqrtnT} sets $c=5$ and $t=3$ initially.  In the first iteration of the while-loop, it can select  $v=1$, implying $u=4$.  Deleting $u=4$ gives that $T'$ is the path on vertices $5,\dots,9$, so the while-loop terminates.  Then, vertex 7 is added to $S$, so $S=\{4,7\}$ and $\ptp(P_9;S)=3=\sqrt 9$.  Now assume vertex 4 is known to be blue (i.e., $X=\{4\}$).  Modified Algorithm \ref{alg2sqrtnT}   described in Remark \ref{thcT2sqrtn-mod} sets $n'=9-1=8$, $t=\lf\sqrt 8\rf = 2$, and $c=5$ (or $c=6$)  initially.  It can select  $v=1$ in the first iteration of the while-loop, implying $u=3$.  Deleting $u=3$ yields $T'$ to be the path on vertices $4,\dots,9$ but vertex $4$ is ignored, so  we have $r=2=t$ and the while-loop terminates.  Then, vertex 7 is added to $S$, so $S=\{3,4,7\}$ and $\ptp(P_9;S)=2=\lf\sqrt 8\rf$.
\end{ex}
 
\begin{thm}\label{thcU3sqrtn} Let $G$ be a unicyclic graph of order $n\ge 3$. Then 
\[\thc(G)\le \thp(G)\le\sqrt 6\sqrt {n}+\frac 1 2.\]
\end{thm}
\bpf 
Let $C_k$ denote the one cycle of $G$ (with $k$ being its order).  Contract the edges of $C_k$ to obtain a tree $T$ of order $n-k+1$; call the contracted vertex $x$.  First we describe a strategy for PSD zero forcing.  Choose a set $S_C\subseteq V(C_k)$  so that $\thp(C_k;S_C)=\thp(C_k)$.  We apply  Modified Algorithm \ref{alg2sqrtnT}  described in Remark \ref{thcT2sqrtn-mod} with the assumption that $X=\{x\}$.  We  choose a set $S_T\subseteq V(T)$  such that $x\in S_T$, $|S_T|\le 1+\sqrt{n-k}$ and $\ptp(T;S_T)\le \sqrt{n-k}$.   Then by forcing only on $C_k$, we can color $C_k$ blue such that $|S_C|+\ptp(C_k;S_C)=\thp(C_k)$. At this point, begin forcing the rest of $G$ using all currently blue vertices, including $S_T\setminus\{x\}$.    The remaining vertices can be colored blue in at most $\ptp(T;S_T)$ time-steps.  Thus,
\bea \thc(G)\le \thp(G)&\le& |S_C|+|S_T|-1 + \ptp(C_k;S_C)+\ptp(T;S_T)\\
&=&\thp(C_k)  +\lf \sqrt{n-k}   \rf+\lf \sqrt{n-k}   \rf\le  \lc \sqrt{2k}-\frac{1}{2}   \rc+ 2\lf \sqrt{n-k}   \rf\\
&\le & \sqrt{2k}+\frac{1}{2}   + 2 \sqrt{n-k}.\eea

To complete the proof, we show that $\sqrt{2k}+\frac{1}{2}   + 2 \sqrt{n-k}\le \sqrt 6\sqrt {n}+\frac 1 2$ for $1\le k\le n$.  The function $f(k)=\sqrt{2k}+\frac{1}{2}   + 2 \sqrt{n-k}$ is differentiable on $[1,n]$ so we can find the maximum by evaluating $f(1), f(n)$, and $f(k_o)$ for $k_o$ such that $f'(k_o)=0$.  
Since $f'(k)= \frac{1}{\sqrt{2} \sqrt{k}}-\frac{1}{\sqrt{n-k}}$, $k_o=\frac n 3$ and $f(\frac n 3) = \sqrt 6\sqrt n+\frac  1 2 $.
Observe that $f(1)=\sqrt{2}+\frac{1}{2}   + 2 \sqrt{n-1}\le \sqrt 6\sqrt {n}+\frac 1 2$ 
and $f(n)=\sqrt{2n}+\frac{1}{2}   \le \sqrt 6\sqrt {n}+\frac 1 2$.
\epf\vspace{-10pt}


\subsection{Girth and degree lower bounds for cop-throttling number}

In this section we present a lower bound on the cop-throttling number in terms of the girth of a graph. We also obtain a minimum degree lower bound from the known lower bound on cop number.  We begin with a technical lemma. 

\begin{lem}[Girth Lemma]\label{lem:girth} 
Let $G$ be a graph with girth $g\geq 3$ and let $C$ be a $g$-cycle of $G$.  Given vertex $v\notin C$, there exists a vertex $x\in V(C)$ such that $\dist(x,u)\leq \dist(v,u)$ for all $u\in C$.  Specifically, we can choose such an $x\in V(C)$ as follows: If $\dist(v,C)\ge  \lf \frac g 2\rf$, then choose any vertex $x\in V(C)$.  If $\dist(v,C)< \lf \frac g 2\rf$, then choose $y\in V(C)$ such that $\dist(v,y)=\dist(v,C)$ and choose $x\in V(C)$ such that $\dist(x,y)=\dist(v,y)$.   \end{lem}

\bpf   Note that for any $u,u'\in C$, $\dist(u,u')\leq \lf g/2\rf$. So for  $v\notin C$ with $\dist(v,C)\geq \lf g/2\rf$, any vertex $x\in C$ has the property that $\dist(x,u)\leq \dist(v,u)$ for all $u\in C$. 

Now assume $\dist(v,C)< \lf \frac g 2\rf$ and choose $y\in V(C)$ such that $\dist(v,y)=\dist(v,C)$ and choose $x\in V(C)$ such that $\dist(x,y)=\dist(v,y)$.  By the choice of $y$, there is no other vertex of $C$ on any shortest path from $v$ to $y$.  Since the length of $C$ is the minimum length of a cycle in $G$, there is a shortest path from $u$ to $u'$ that lies entirely within $C$  for any $u,u'\in V(C)$. 

We first show that $\dist(x,w)\le \dist(v,w)$ for $w\in W$, where $W$ is the set of all $ w\in V(C)$ such that some shortest path from $v$ to $w$ contains no other vertex of $C$.     The result is clear for $w=y$, so assume $w\ne y$.  
If $w$ is on a shortest path from $x$ to $y$, then $\dist(x,w)<\dist(x,y)=\dist(v,y)\le \dist(v,w)$.  
So assume no shortest path from $x$ to $y$  contains $w$.  
There are two paths in $C$ from  $w$ to $y$; let $P$ denote the path that does not contain $x$.  Consider the cycle $C'$ formed by the union of a shortest path from $y$ to $v$, a shortest path from $v$ to $w$ that has no other vertices of $C$, 
and $P$.  Observe that $C$ itself can be viewed as  the union of a shortest path from $y$ to $x$, a path $Q$ from $x$ to $w$ that does not contain $y$, and $P$.  Since $\len(C')\ge \len(C)$,
$ \dist(y,v)+\dist(v,w)+\len(P)\ge \dist(y,x)+\len(Q)+\len(P)\ge \dist(y,x)+\dist(x,w)+\len(P)$, where $\len(G)$ denotes the length of a path or cycle $G$.  Since $\dist(y,x)=\dist(y,v)$, this reduces to 
$\dist(v,w)\ge\dist(x,w)$.

Now, consider an arbitrary $u\in V(C)$ and let $w$ be the first vertex in $C$ on a shortest path from $v$ to $u$.  Then, $\dist(v,u)=\dist(v,w)+\dist(w,u)\ge\dist(x,w)+\dist(w,u)\ge \dist(x,u)$.
\epf

Observe that the Girth Lemma could have been stated without the hypothesis that $v\not\in V(C)$, because $x=v$ works for  $v\in V(C)$.   Given a cycle $C$ with length equal to the girth of $G$  and  $v\in V(G)$, we use the notation $x(v)$ to denote a vertex $x$ with the property that $\dist(x,u)\le\dist(v,u)$ for all $u\in V(C)$.  We now use the Girth Lemma to prove the next theorem.
 
\begin{thm}\label{thm:girth}
Let $G$ be a graph with girth $g\geq 3$. Then, $\thc(G)\geq \lc \sqrt{2g}-\frac{1}{2}\rc$, and this bound is tight.
\end{thm}
\begin{proof}
Let $C$ be a $g$-cycle of $G$.
Suppose the robber's strategy is to start in a vertex of $C$ and to remain inside $C$ until  caught (and that this is known to the cops). We claim that for this strategy of the robber, $k$ cops and $p$ rounds %
 will be required to catch the robber, where $k+p=\thc(G[V(C)])=\thc(C_{g})=\lc \sqrt{2g}-1/2\rc$. %
  Indeed, no advantage is offered to the cops by the existence of vertices outside $C$, since  a cop who starts at a vertex $v$ outside $C$  could not reach a vertex  in $C$ any faster than a cop who starts  at the vertex $x(v)$ in $C$ whose existence is guaranteed by  the Girth Lemma. Since $k$ cops can catch the robber in $p$ rounds in $G[V(C)]$, and since the cop-throttling number of $G$ is at least the cop-throttling number of $G$ with a fixed strategy for the robber, it follows that $\thc(G)\geq \lc \sqrt{2g}-\frac{1}{2}\rc$. The bound is clearly tight, e.g., for $G= C_n$, $n\geq 3$.
\end{proof}

Note that the following connections between girth, minimum degree,  and cop-number are known.

\begin{thm}{\rm \cite{AF84}, \cite[Theorem 1.3]{CRbook}}  If $G$ is a graph of girth at least $5$, then $c(G)\ge \delta(G)$.
\end{thm}

\begin{cor}  If $G$ is a graph of girth at least $5$, then $\thc(G)\ge \delta(G)$.
\end{cor}


\subsection{Connections with the burning number of a graph}

\emph{Graph burning} is a deterministic process on a simple graph in which one seeks to \emph{burn} every vertex of the graph in a minimum number of steps. In step 1, a vertex is chosen to be burned; in each subsequent step, a new unburned vertex is chosen to be burned, and in addition, every unburned neighbor of a burned vertex becomes burned. The process continues until the entire graph is burned. This produces a burning sequence $(v_1, v_2, \ldots v_m)$, 
where $v_i$ is the vertex chosen to be burned in the $i^{th}$ step. The  \emph{burning number} of a graph $G$ is the minimum number $m$ for which a burning sequence of length $m$ exists, and is denoted by $b(G)$. 

The burning number was introduced in \cite{bonato2014burning} as a graph invariant that could be used to model the spread of a contagion in the graph. It was shown in that paper that $b(G) = \min\{b(T) : T \mbox{ is a spanning tree of } G\}$  and $b(P_n) = \lf \sqrt{n} \rf$. It was conjectured there that for any graph $G$ on $n$ vertices, $b(G) \leq \lf \sqrt{n} \rf$. However, the best known upper bounds are $b(G) \leq \sqrt{\frac{12n}{7}} + 3$ (see \cite{bessy2015bounds}), and $b(G) \leq \lc\frac{\sqrt{24n+33}}{4} - \frac{3}{4}\rc \approx \frac{\sqrt{6}}{2}\sqrt{n}$ (see \cite{land2016upper}).

This invariant $b(G)$ is related to PSD zero forcing (and hence, cop-throttling) in the case where $G$ is a  tree. We can consider a burning process with burning sequence $(v_1, v_2, \ldots v_m)$ as a `staggered' PSD zero forcing process, i.e., one in which we are only permitted to color one vertex from the zero forcing set at a time. PSD color changes  propagate through a tree in the same manner as a burning process:  Every neighbor of a blue vertex in a tree will be forced in the next step in a PSD zero forcing process. This is not the case on graphs with cycles. 
\begin{prop}\label{prop:burn}
Let $T$ be a tree and $b(T)$ its burning number. Then $\thc(T)=\thp(T) \leq 2b(G) - 1$ and this bound is tight.
\end{prop}
\bpf
Let $b(T) = m$ and $(v_1, v_2, \ldots, v_m)$ a burning sequence for $T$. Consider a PSD zero forcing process with $S = \{v_1, \ldots, v_m\}$ as the set of vertices initially colored blue. Then $\ptp(T; S) \leq m-1$; hence, $\thp(T) \leq \thp(T; S) \leq m + (m-1) = 2m-1$. 

To see that the bound is tight, consider the tree $T$ shown in Figure \ref{fig:burnex}.  The burning sequence (2,4) and $|V(T)|>1$ imply $b(T)=2$, whereas $\gamma(T)=2$ and  $|V(T)|>2$ imply $\thp(T)=3$. \vspace{-10pt}
\begin{figure}[h!] \begin{center}
\scalebox{.5}{\includegraphics{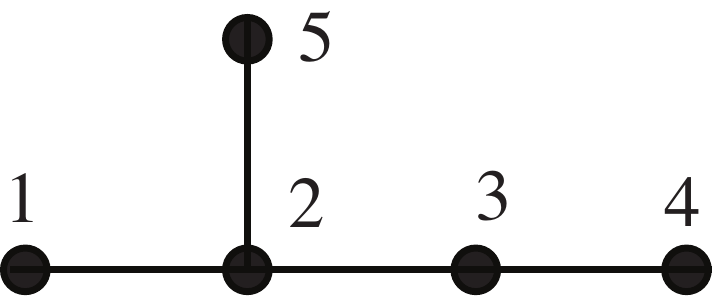}}
\caption{\label{fig:burnex} A tree $T$ for which $\thc(T)=\thp(T)=3=2b(G)-1$ \vspace{-15pt}}
\end{center}
\end{figure}
\epf

 We remark that the above bound is usually rather poor, as delaying the start of forcing in a PSD zero forcing set is inefficient.  As an example, consider the path: $b(P_n) = \lf \sqrt{n} \rf$, so $2b(G) - 1=2\lf \sqrt{n} \rf-1$, whereas $\thc(P_n) = \thp(P_n) = \lc \sqrt{2n} - \frac{1}{2} \rc.$  Nevertheless, the relationship in Proposition \ref{prop:burn} is interesting, and together with  $b(G) \lesssim \frac {\sqrt 6}2 \sqrt n$  readily produces the upper bound $\thc(T) \lesssim \sqrt 6\sqrt{n}  = O(\sqrt{n})$ for trees.


\section{Low throttling number}\label{sextremeL} 

In this section, we characterize the graphs with $\thc(G)=h$ for fixed low values of $h$. We do so by partitioning $h$ as $(k,p)$, where $k$ is interpreted as a number of cops, $p$ is interpreted as a number of rounds, and $h=k+p$.  
\begin{obs} \label{thc=1} For a graph $G$, 
$\thc(G)=1$ if and only if $G=K_1$.
\end{obs}\vspace{-8pt}

 \begin{prop}
\label{thc=2} For a graph $G$, 
$\thc(G)=2$ if and only if $|V(G)|\geq 2$ and $G$ satisfies  one of the following conditions: \vspace{-3pt}
\begin{enumerate}[(1)]
\item $\gamma(G)=1$.\vspace{-3pt}
\item $G=2K_1$.\vspace{-3pt}
\end{enumerate}
\end{prop}
\bpf
If $\thc(G)=2$, there are two possibilities for $(k,p)$, namely $(k,p)=(2,0)$ or $(k,p)=(1,1)$. It is easy to see that $k=2$ and $\capt_2(G)=0$ is possible if and only if $G=K_2$ or $G=2K_1$.
Likewise, $k=1$ and $\capt_1(G)=1$ is possible  if and only if $G$ has a dominating vertex and order at least two.
\epf\vspace{-8pt}

\begin{prop}\label{thc=3} For a graph $G$,  $\thc(G)=3$ if and only if $|V(G)|\ge 3$ and $G$ {satisfies 
one} of the following conditions:
\begin{enumerate}[(1)]
\item \label{12} $\gamma(G)>2$ and there exists $z\in V(G)$ such that
 \ben
 \item \label{12a} for all $v\in V(G)$, $\dist(z,v)\le 2$, and 
 \item \label{12b} for all $w\in G-N[z]$, there is a vertex $u\in N[z]$ such that $N[w]\subset N[u]$.
 \een
\item\label{21} $\gamma(G)=2$.
\item \label{30} $G=3K_1$.
\end{enumerate}
\end{prop}
\bpf Assume $\thc(G)=3$, which implies $|V(G)|\ge 3$. There are three possible ways $\thc(G)=3$ can be achieved, namely $(k,p)=(3,0), (2,1),$ or $(1,2)$.  
The hypotheses $(k,p)=(3,0)$ and $|V(G)|\ge 3$ imply $|V(G)|=3$, which together with $\thc(G)=3$ implies   $G=3K_1$ or $G=K_2\du K_1$; observe that $\gamma(K_2\du K_1)=2$.
The hypotheses $(k,p)=(2,1)$ and $\thc(G)=3$ imply $\gamma(G)=2$. 

So we assume $\gamma(G)> 2$ and $G\ne 3K_1$, implying there is  a vertex $z\in V(G)$ such that $\capt(G,\{z\})=2$. This implies every vertex is  distance at most 2 from $z$.  The robber must avoid the closed neighborhood of $z$ to avoid capture in one round, so assume the robber chooses $w \in G -N[z]$.  There must be a way for the cop to catch the robber in two rounds.  That implies there is a vertex $u$ adjacent to $z$  such that $N[w]\subset N[u]$.  

For the converse, it is clear that either  of $G=3K_1$ or ($\gamma(G)=2$ and $|V(G)|\ge 3$) implies  $\thc(G)=3$. So assume $|V(G)|\ge 3$,  $\gamma(G)>2$, and there exists $z\in V(G)$ such that
 for all $v\in V(G)$, $\dist(z,v)\le 2$, and 
 for all $w\in G-N[z]$, there is a vertex $u\in N[z]$ such that $N[w]\subset N[u]$. It is clear $|V(G)|\ge 3$, $G\ne 3K_1$, and $\gamma(G)>2$.  The conditions $|V(G)|\ge 3$ and  $\gamma(G)>2$ imply $\thc(G)\ge 3$. The existence of $z\in V(G)$ satisfying conditions \eqref{12a} and \eqref{12b} 
 implies that $\capt(G,\{z\})\le 2$, so $\thc(G)\le 3$. \epf
 
 The stellated wheel discussed in Example \ref{ex:SW} illustrates an instance of  condition \eqref{12}. 
 
 Hahn and MacGillivray \cite{HahnMacG} provide a polynomial-time algorithm for determining  whether $\text{capt}_k(G)\leq p$ for a given graph $G$,  given that the capture time $p$ and the number of cops $k$ are fixed. We will call this ``Algorithm HM." In Algorithm \ref{alg2} below, we give a polynomial-time algorithm for determining whether 2 cops can catch a robber in 2 steps. A proof of the correctness of Algorithm \ref{alg2} is included in the proof of Theorem \ref{thc=4}; see \cite{alg2code} for an implementation of Algorithm \ref{alg2} in M{\small ATLAB}.\vspace{5pt}

\begin{algorithm2e}[H]
\caption{Can 2 cops catch the robber in 2 steps?\label{alg2}}

\textbf{given} $G=(V,E)$\;
\For{$\{c_1,c_2\}\subset V$}{
sum$\gets$0\;
\For{$r\notin N[\{c_1,c_2\}]$}{
flag$\gets$0\;
\For{$c_1'\in N[c_1]$}{
\For{$c_2'\in N[c_2]$}{

\If{$N[r]\subseteq N[\{c_1',c_2'\}]$}{
flag$\gets$1\;
}
}
}
sum$\gets$sum+flag;
}
\If{$\emph{sum}=|V|-|N[\{c_1,c_2\}]|$}{
\Return {\bf true}  (i.e., 2 cops at  $\{c_1,c_2\}$ always catch the robber in 2 steps)
}
}
\Return {\bf false}  (i.e., the robber can always evade 2 cops for more than 2 steps)
\end{algorithm2e}

\begin{thm}\label{thc=4}  For a graph $G$, 
$\thc(G)=4$ if and only if $\thc(G)\notin \{1,2,3\}$ according to    Propositions \ref{thc=2}, \ref{thc=3}, and Observation \ref{thc=1}, and $G$ satisfies at least one of the following conditions:
\begin{enumerate}[(1)]
\item \label{13} $G$ is a graph such that Algorithm HM returns true for $\emph{capt}_1(G)\leq 3$.\vspace{-4pt}
\item \label{22} $G$ is a graph such that Algorithm \ref{alg2} returns true.\vspace{-4pt}
\item \label{31} $\gamma(G)=3$.\vspace{-4pt}
\item \label{40} $G=4K_1$.\vspace{-4pt}
\end{enumerate}
\end{thm}

\bpf It is easy to see that if $\thc(G)\geq 4$ according to Propositions \ref{thc=2}, \ref{thc=3}, and Observation \ref{thc=1}, and if $G$ satisfies any of the four conditions in the statement of the theorem, then $\thc(G)= 4$. 

    Now assume that $\thc(G)=4$. There are four possible ways this can be achieved, namely $(k,p)\in\{(4,0), (3,1), (2,2), (1,3)\}$. The case $(k,p)=(4,0)$ implies $V(G)=4$; it can be verified that all graphs on 4 vertices  that do not have cop-throttling number less than 4 (according to Propositions \ref{thc=2}, \ref{thc=3}, and Observation \ref{thc=1}) satisfy at least one of the conditions in the statement of the theorem. The case $(k,p)=(3,1)$ implies $\gamma(G)=3$. The case $(k,p)=(1,3)$ implies 1 cop can catch the robber in 3 steps, so $G$ must be a graph such that Algorithm HM returns `true' for $\text{capt}_1(G)\leq 3$ (but returns `false' for $\text{capt}_1(G)\leq 2$, i.e., does not have cop-throttling number less than 4).

Finally, the case $(k,p)=(2,2)$ implies 2 cops can catch the robber in 2  steps. 
 We  show that Algorithm \ref{alg2} correctly answers the question ``can 2 cops catch the robber in 2 steps?''. The answer to this question is `yes' if and only if there  is a multiset $\{c_1,c_2\}$ of cop starting positions such that for any starting position $r$ of the robber, there are cop positions $c_1'\in N[c_1]$ and $c_2'\in N[c_2]$ such that $N[r]\subseteq N[\{c_1',c_2'\}]$. In other words, for every move the robber can make, capture can occur in the following step. Note that the robber cannot start in $N[\{c_1,c_2\}]$, since capture will occur in one step. In Algorithm \ref{alg2}, the counter `sum' increments each time there exist such  $c'_1$ and $c'_2$ for a feasible starting position of the robber $r$. If after iterating through all feasible starting positions of the robber, $\text{sum}=|V|-|N[\{c_1,c_2\}]|$, it follows that the robber will always be caught by cops starting at positions $\{c_1,c_2\}$, and  Algorithm \ref{alg2} returns  `true.' Otherwise, if after iterating through all initial cop positions $\{c_1,c_2\}$, some robber position $r$ always allows the robber to evade the cops for 2 steps, then Algorithm \ref{alg2} returns  `false.' Thus, Algorithm \ref{alg2} correctly determines whether 2 cops can  catch the robber in 2 steps.
\epf
\vspace{-10pt}

\section{High throttling number}\label{high}

In this section, we discuss the question of how large $\thc(G)$ can be for a given graph $G$. We have shown in Section \ref{sPSDZF} that  $\thc(P_n) = \Theta(\sqrt{n})$, $\thc(C_n) = \Theta(\sqrt{n})$, and $\thc(G) = O(\sqrt{n})$ for any tree or unicyclic graph $G$ on $n$ vertices. 

{\begin{quest}\label{Q:thc-sqrtn}
Asymptotically, what is the largest possible $\thc(G)$ for graphs $G$ of order $n$? In particular, is $\thc(G) = O(\sqrt{n})$? 
\end{quest}
}

We investigate some other classes of graphs for which capture time or cop number is known to be quite large, and show that for all these classes, Question \ref{Q:thc-sqrtn} is answered in the affirmative. Recall that the famous  Meyniel's Conjecture  states that the upper bound on the cop number of any  connected graph {of order $n$} is $O(\sqrt{n})$, and that this is best possible in that there exist infinite families of  connected  graphs $\{G_n\}$ such that for sufficiently large $n$, there is a constant $d$ for which $c(G_n) \geq d\sqrt{n}$ \cite[Chapter 3]{CRbook}. Such a family is called a \emph{Meyniel extremal family}.  We again remark that a proof of a   %
positive answer to Question \ref{Q:thc-sqrtn} would prove Meyniel's Conjecture, because $c(G)  \le \thc(G)$.


\subsection{Throttling number of a Meyniel extremal family}
One example of a Meyniel extremal family is the family of incidence graphs of finite projective planes of order $q$, denoted $\mathcal{P}_q$, where $q$ is a prime power.  A {\em finite projective plane} $\mathcal{P}_q$ has $q^2 + q + 1$ points. For every two distinct points, there is a unique line passing through them, and for every two distinct lines, there is a unique point that lies on both lines. Finally, there are four points such that no line is incident with more than two of them. The {\em incidence graph $\incpp q$} of a finite projective plane $\mathcal{P}_q$ is a bipartite graph on $2(q^2+q+1)$ vertices. In one part, the vertices represent points of $\mathcal{P}_q$ and in the other, the vertices represent the lines of $\mathcal{P}_q$. A vertex in the first part is adjacent to a vertex in the other if the corresponding point lies on the corresponding line.

It is known that the cop number of $\incpp q$ is $q+1$ (see \cite{praalat2010does}, \cite[Thm 1.5]{CRbook}), and so if Meyniel's Conjecture is true, then the bound $c(G) = O(\sqrt{n})$ is asymptotically tight. 
Although the number of cops required to catch the robber is large for $\incpp q$, the capture time for $q+1$ cops on $\incpp q$ is very low: A winning cop strategy is given in \cite[page 294]{praalat2010does} for $q+1$ cops that takes only 3 rounds, so $\thc(\incpp q)\le \sqrt{|V(\incpp q)|}+3$.

Note that the only orders for which finite projective planes (and hence the incidence graphs $\incpp q$) are known to exist are for prime powers, and that a Meyniel extremal family should consist of an infinite family of graphs of every order $n$. A description is given in \cite{praalat2010does} (see also \cite[Thm 3.8]{CRbook}) of how to construct a graph of any order $n$ from these incidence graphs by adding a path. That is, choose $q$ such that $2(q^2+q+1) \leq n$, and form the graph $G_n$ by adding a path of length $n-2(q^2+q+1)$ to $\incpp q$, producing a graph of order $n$ that is easily seen to  have $c(G_n) = q+1$. 
 By Bertrand's postulate \cite[Theorem 3.7]{CRbook} (also known as Chebyshev's theorem), %
 for any integer $x\ge 2$ there is a prime in the interval $(x, 2x)$. Thus, a prime $q$ can be chosen in the interval $\left(\lf\sqrt{\frac n 8}-1\rf,2\lf\sqrt{\frac n 8}-1\rf\right)$, which implies $\lf\sqrt{\frac n 8}\rf\le q+1=c(G_n)\le \lf\sqrt{\frac n 2}\rf$. Therefore,  $c(G_n) = \Theta(\sqrt{n})$.

\begin{prop}
Let $\{G_n\}$ denote the Meyniel extremal family described above. Then $\thc(G_n) = \Theta(\sqrt{n})$. 
\end{prop}

\begin{proof}
Fix $n$ and choose $q$ as described above to form the graph $G_n$ as a vertex sum of $\incpp q$ and a path of order $k = n-2(q^2 + q+1)+1$. 
By Corollary \ref{cor:cliquesum},
\begin{eqnarray*}
\thc(G_n) & \leq & \thc(P_k) + \thc(\incpp q) \\
& \leq & \lc \sqrt{2k} - \tfrac{1}{2} \rc + \sqrt{n-k} + 3 \\
& = & O(\sqrt{n}).
\end{eqnarray*}
Furthermore, since $c(G_n) = q+1 = \Omega(\sqrt{n})$ and $\thc(G_n) \geq c(G_n)$, we have $\thc(G_n) = \Theta(\sqrt{n})$.
\end{proof}


\subsection{Throttling number for a family of cop-win graphs with maximum capture time}

We remark that   a graph  $G$ with large capture time for the minimum number $c(G)$ of cops required to catch the robber   does  not necessarily have high throttling number, since the introduction of more cops can greatly decrease the capture time. 
In \cite{BGHK09}, an upper bound for the capture time of a single cop on a cop-win graph $G$ is given as $\capt(G) \leq n-4$, where $n$ is the order of the graph. An infinite family of graphs attaining this bound is also given. In \cite{gavenvciak2010cop}, a characterization of all cop-win graphs with capture time $n-4$ is given, along with a family of such graphs with the minimum number of edges; this family is described in the next example. 

\begin{figure}[h!]
\begin{center}
		\scalebox{.45}{\includegraphics{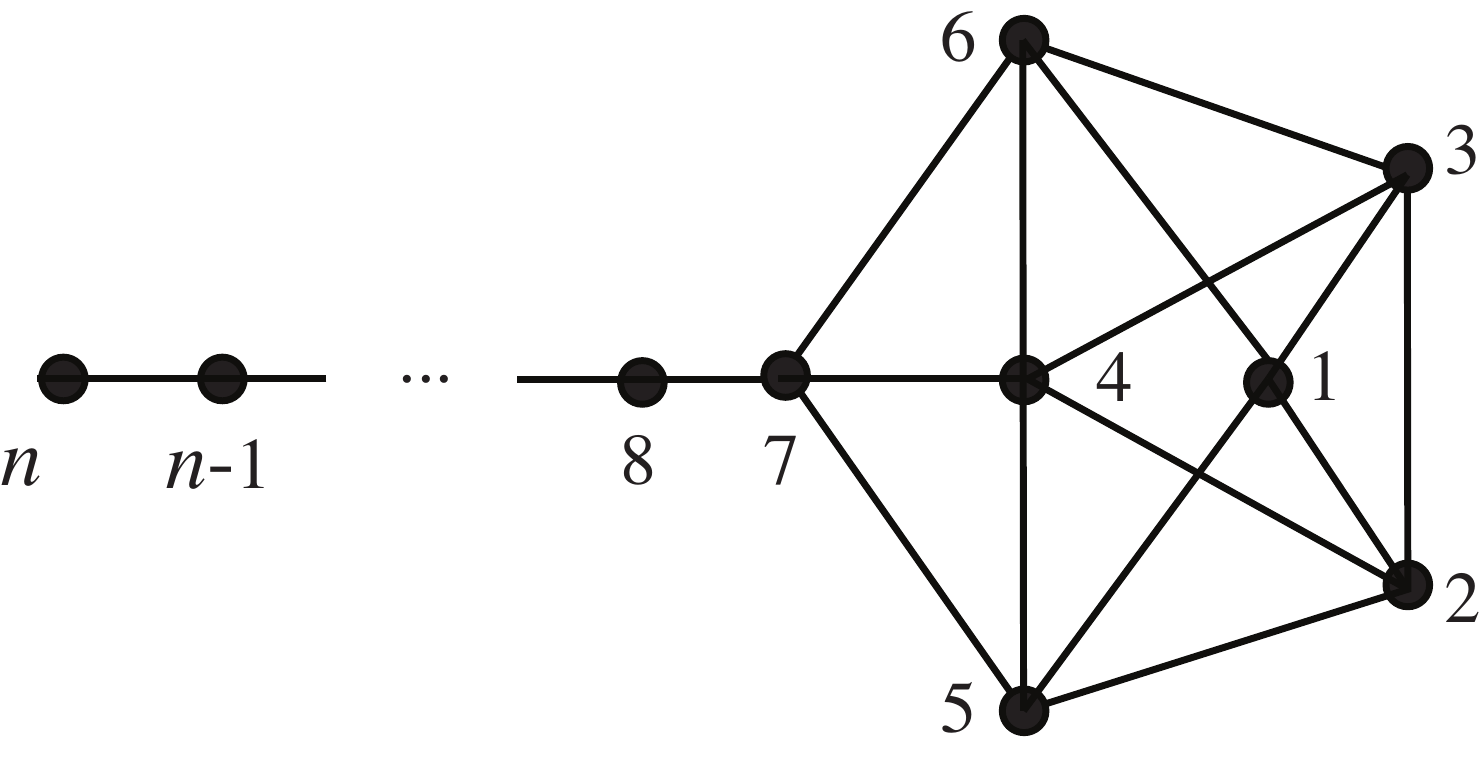}}
\caption{ A family $\{H_n : n \geq 7\}$ of cop-win graphs with maximum capture time \cite{gavenvciak2010cop}} \label{fig:Hn}
\end{center}\end{figure}\vspace{-15pt}

\begin{ex}\label{ex:Hn} The graph  $H_n$ is shown Fig.~\ref{fig:Hn}.  It is shown in \cite{gavenvciak2010cop} that $\capt(H_n)=n-4$.
  The throttling number of $H_7$ is equal to $3$:  The fact  $\{1, 7\}$ is a  dominating set of order 2 implies that $\thc(H_7) \leq 3$.    Since the capture time with 1 cop is $7-4 = 3$, $\thc(H_7)\ne 2$. 
Since $H_n$ is the vertex sum of $P_{n-6}$ and $H_7$, Theorem \ref{cliquesum} implies
\[\lc \sqrt{2n-12} - \tfrac{1}{2}\rc=\thc(P_{n-6})\le \thc(H_n) \le \thc(P_{n-6})+\gamma(H_7)= \lc \sqrt{2n-12} - \tfrac{1}{2}\rc + 2, \]
so $\thc(H_n)=\Theta(\sqrt{n})$. 
\end{ex}

\subsection{Upper bounds on throttling for hypercubes and grids}

The last classes of graphs that we explore in this section are the hypercubes $Q_m$ with $n = 2^m$ vertices and grids $P_a \Box P_b$ with $n=ab$ vertices. It is shown in \cite{BPPR17} that when studying the asymptotics of the $k$-capture time on the hypercube $Q_m$, the number of cops must be increased greatly before it has an effect on the order of the capture time. In particular, it is known that $c(Q_m) = \lf \frac{m+1}{2}\rf$ (see \cite{CRbook}), and that the capture time is $\Theta(m\log m)$ (see \cite{bonato2013capture}). Furthermore, it is shown in \cite{BPPR17} that for all $k \leq 2^{m^\alpha}$ where $\alpha < 1$, $\capt_k(Q_m) = \Theta(m\log m).$ However, 
\[\thc(Q_m) \leq \lf \frac{m+1}{2} \rf + m\log m,\]
which is certainly less than $\sqrt{n} = 2^{m/2}$. That is, $\thc(Q_m) = O(\log n \log\log n)$. 

Similarly, it was shown in \cite{grid_capture_time} that $c(P_a\Box P_b)=2$, and that the capture time is $\lf\frac{a+b}{2}\rf-1$. Thus, $\thc(P_a\Box P_b)\leq \lf\frac{a+b}{2}\rf+1=O(\sqrt{n})$. Note that there are also other ways to achieve cop-throttling of $O(\sqrt{n})$, i.e., there are other pairs $(k,p)$ such that $k$ cops can catch a robber in $p$ steps and $k+p=O(\sqrt{n})$. For example, consider a standard lattice labeling of the vertices of $P_a\Box P_b$, $a\geq b\geq 2$, (where the four vertices of degree 2 have labels $(1,1), (a,1), (1,b), (a,b)$).  If the cops start at the vertices with coordinates $(\lf a/2\rf,2), (\lf a/2\rf,4), \ldots, (\lf a/2\rf,b-1)$ and move along the $a$-axis toward whichever half of the grid  contains the robber, they can always catch the robber in $a/2+O(1)$ moves.  Thus, this strategy also yields  $\thc(P_a\Box P_b)=O(\sqrt{n})$.  


\section{Conclusion}

In this paper, we introduced and explored the concept of throttling for the game of Cops and Robbers. We presented several methods for bounding the cop-throttling number of general graphs using parameters such as the PSD zero forcing number, burning number, girth, and minimum degree. We also investigated cop-throttling in trees and unicyclic graphs, characterized graphs with low cop-throttling numbers, and explored how large the cop-throttling number can be in general. All bounds obtained suggest that $\thc(G)=O(\sqrt{n})$; if this asymptotic bound holds in general, it would imply Meyniel's Conjecture.

One direction for future work related to cop-throttling is to determine whether there exists a family of trees for which the cop-throttling number is asymptotically equal to $2\sqrt{n}$; similarly, it would be interesting to determine if the bound in Theorem \ref{thcU3sqrtn} for unicyclic graphs is tight, or if it can be improved. Further, using the clique sum result given in Theorem  \ref{cliquesum} and the result of Theorem \ref{thcU3sqrtn}, one could perhaps obtain a bound for the cop-throttling number of cactus graphs (graphs in which any two cycles share at most one vertex). It would also be interesting to develop an exact algorithm for cop-throttling of cactus graphs, and more generally, bounds for planar graphs; some results on the cop number and capture time of planar graphs (e.g. \cite{AF84,loh2017,pisa2016}) may be useful toward this end.\vspace{-5pt}


\subsection*{Acknowledgements}  We thank Anthony Bonato for stimulating discussions related to this project.  We thank  the  National Science Foundation  (NSF 1604458, NSF 1604697), the Combinatorics Foundation, and the International Linear Algebra Society for financial support of this project.  


\end{document}